\documentclass{amsart}

\usepackage{amsmath}
\usepackage{amsthm}
\usepackage{amssymb}
\usepackage{amsfonts,mathrsfs}
\usepackage{stmaryrd}
\usepackage{amsxtra}  
\usepackage{epsfig}
\usepackage{verbatim}
\usepackage{xcolor}

\theoremstyle{plain}
\newtheorem{theorem}[equation]{Theorem}

\newtheorem{lemma}[equation]{Lemma}
\newtheorem{corollary}[equation]{Corollary}
\theoremstyle{definition}

\newtheorem{example}[equation]{Example}
\theoremstyle{remark}
\newtheorem{remark}[equation]{Remark}

\newcommand{\dbar}{\bar \partial}

\newcommand{\supp}{\text{supp}\,}
\newcommand{\coloneqq}{\mathrel{\mathop:}=}
\newcommand{\eqqcolon}{=\mathrel{\mathop:}}
\DeclareMathOperator{\dist}{dist}

\begin{document}

\title[Dimension of the Bergman space]{On the dimension of the Bergman space for some unbounded domains}
\author{A.-K. Gallagher, T. Harz \& G. Herbort}
\thanks{Special thanks to the Institute of Mathematics, particularly Professor Ines Kath, at the University of Greifswald for their hospitality and support of the first author during  the academic year 2015--16. The second author was supported by the National Research Foundation of Korea (NRF) grant funded by the Korean government (MSIP) (No. 2011-0030044).}
\address{Department of Mathematics, Oklahoma State University, Stillwater, USA}
\email{anne-katrin.herbig@okstate.edu}
\address{Department of Mathematics, POSTECH, Pohang, Korea}
\email{tharz@postech.ac.kr}
\address{Department of Mathematics, University of Wuppertal, Wuppertal, Germany}
\email{gregor.herbort@math.uni-wuppertal.de}
\begin{abstract}
A sufficient condition for the infinite dimensionality of the Berg\-man space of a pseudoconvex domain is given. This condition holds on any pseudoconvex domain that has at least one smooth boundary point of finite type in the sense of D'Angelo.
\end{abstract}
\maketitle
\section{Introduction}
Let $\Omega$ be a domain in $\mathbb{C}^{n}$, $n\geq 1$. The Bergman space, $A^{2}(\Omega)$,  is the space of holomorphic functions, $\mathcal{O}(\Omega)$, on $\Omega$ which belong to $L^{2}(\Omega)$, i.e.,
  $$A^{2}(\Omega)= \big\{h \colon \Omega\to \mathbb{C}\;\text{holomorphic}\;:
  \int_{\Omega}|h|^{2} dV<\infty \big\},$$
where $dV$ denotes the Euclidean volume form. In this article, it will be shown that the dimension of the Bergman
space is infinite for a large class of unbounded pseudoconvex domains in $\mathbb{C}^{n}$.

The Bergman space of a domain in the complex plane was proven  to be either $0$ or infinite dimensional by Wiegerinck in \cite{Wie84}.
Moreover, a result of Carleson \cite[Theorem 1.a in \S VI]{Car67} depicts this dichotomy in terms of the logarithmic capacity of 
the complement of the given domain. That is, for any domain $\Omega\subset\mathbb{C}$, $A^{2}(\Omega)$ is 
non-trivial if and only if $\Omega^{c}$ has positive logarithmic capacity. (We note that Carleson's result in \cite{Car67} is stated for domains with compact complement, but it still holds true without this additional assumption.)

The dimension of the Bergman space is invariant under biholomorphic transformations. Indeed, each biholomorphic map $F\colon \Omega \to \Omega'$ induces an isometric isomorphism $A^{2}(\Omega') \to A^{2}(\Omega)$ via the assignment
$f \mapsto (f\circ F)\cdot\det(J_{\mathbb{C}}F)$, where $J_{\mathbb{C}}F = (\partial F_j / \partial z_k)_{j,k=1}^n$. Since the Bergman space of a bounded domain $\Omega\subset\mathbb{C}^{n}$ contains any 
 monomial in the coordinates of $z=(z_{1},\dots,z_{n})$, it follows that the Bergman space of $\Omega$, and  of any domain biholomorphic equivalent to $\Omega$, has infinite dimension. On the other hand, the fact that $A^{2}(\mathbb{C}^{n})=\{0\}$ implies that for any Fatou--Bieberbach domain, i.e., any domain biholomorphic equivalent to $\mathbb{C}^{n}$, the Bergman space is trivial.

Further results for domains in $\mathbb{C}^{n}$, $n\geq 2$, are scattered. In \cite{Wie84}, Wiegerinck 
constructed  non-pseudoconvex domains $\Omega_{k}\subset\mathbb{C}^{2}$, $k\in\mathbb{N}$, so that 
$\dim A^{2}(\Omega_{k})=k$. Whether the above dichotomy property 
holds for  Bergman spaces of pseudoconvex domains in higher dimensions is not known.
In  \cite{Juc12}, Jucha derives this property for the Bergman space of Hartogs domains of the form
$$ \Omega_\phi = \big\{(z,w)\in\mathbb{C}\times \mathbb{C}^{N}:|w|\leq e^{-\phi(z)}\big\},$$ where $\phi$ is subharmonic
on $\mathbb{C}$. Theorem 4.1 therein details the dimensionality of $A^2(\Omega_\phi)$ in terms of properties of the Riesz measure 
of $\phi$. Sufficient conditions for the infinite dimensionality of the Bergman space of further classes of Hartogs domains are 
derived or indicated in \cite[Proposition 0]{Ligocka89}, \cite[pg. 128]{PZ05}, \cite[pg. 22]{Dinew07}, \cite[Corollary 3.3]{Juc12}.
Furthermore, the Bergman space  was shown to be infinite dimensional for a certain unbounded worm domain in \cite[Proposition 1.3]
{KPS14} and non-trivial for special classes of domains in
\cite[Proposition 5.3]{AGK16}.

A general criterion for infinite dimensionality of the Bergman space follows from the proof of Theorem 1 in \cite{CZ02}: If $\Omega \subset \mathbb{C}^n$ admits a bounded continuous function $\varphi$ which is strictly plurisubharmonic on $\Omega$, then $\dim A^{2}(\Omega)=\infty$. A more refined version of this result is also mentioned in \cite{AB15}. Therein the function $\varphi$ needs to be  strictly plurisubharmonic only on some non-empty open subset of $\Omega$ in order to deduce  that $\dim A^{2}(\Omega)=\infty$.
The advantage of the latter formulation is that it allows one to conclude the existence of non-trivial functions in $A^2(\Omega)$ from local information on $b\Omega$. 

\medskip

While the result mentioned in \cite{AB15} is known among experts of the field, a concrete reference in the literature seems to be missing. The first objective of the present paper is to give a short but complete proof of its statement, see Lemma \ref{L:linearlyindependent} and Theorem \ref{T:suffinfinite}. Our version of the result will be slightly more general with respect to properties of the plurisubharmonic functions involved. That is, continuity is not required and the existence of weak singularities is permitted (in some situations this additional generality can be helpful, see, for example, Theorem 1.2 in \cite{Chen13}). The second goal of this article is to use the above result in order to obtain a number of easily checkable conditions for the existence of square-integrable holomorphic functions, see Theorem \ref{T:suffcore} below. In particular, these conditions imply certain properties for pseudoconvex neighbourhoods of domains with trivial Bergman space, see Corollary \ref{C:fatoubieberbach} and Theorem \ref{T:finitedim}.

\medskip

We now state the precise results of the article. For this purpose, let us introduce a family of plurisubharmonic functions on a given domain $\Omega\subset\mathbb{C}^{n}$ as follows:
$$\mathcal{PSH}'(\Omega) \coloneqq \big\{\varphi \colon \Omega \to [-\infty, \infty) \text{ plurisubharmonic},\, \varphi \not\equiv -\infty,\, \nu(\varphi,\,\cdot\,) \equiv 0 \big\}.$$
Here $\nu(\varphi,z)$ denotes the Lelong number of $\varphi$ in $z$. It is straightforward to observe that whenever $\varphi:\Omega\longrightarrow [-\infty,0)$ is plurisubharmonic, then $-\log(1-\varphi)$ belongs to $\mathcal{PSH}'(\Omega)$, see also part (c) of Remark \ref{R:Remark}.
To formulate a general sufficiency condition for the infinite dimensionality of  $A^2(\Omega)$, we consider the following notion of the core, $\mathfrak{c}'(\Omega)$, of $\Omega$, that is,
\begin{align*}
   \mathfrak{c}'(\Omega)\coloneqq \big\{ 
      &z\in\Omega : \text{every } \varphi \in \mathcal{PSH}'(\Omega) \text{ that is bounded from} \\ 
      &\text{above fails to be strictly plurisubharmonic near } z
    \big\}.
\end{align*}
Note that a function $\varphi$ is said to be strictly plurisubharmonic near $z$ if there exists an $\epsilon>0$ such that  $\varphi-\epsilon\left|\,\cdot\,\right|^2$ is plurisubharmonic near $z$.
Similar notions of a core appear in the works of S\l{}odkowski--Tomassini \cite{ST81}, Harvey--Lawson \cite{HL12,HL13} and 
Harz--Shcherbina--Tomassini \cite{HST14}. We note that in general $\mathfrak{c}'(\Omega)$ is different from the analogous core set with
respect to $\mathcal{C}^\infty$-smooth functions, which is considered in \cite{HST14}; see also part (d) of Remark \ref{R:Remark}.

If a point $z\in\Omega$ lies in
the complement of $\mathfrak{c}'(\Omega)$, then there exists a function $\varphi \in \mathcal{PSH}'(\Omega)$ which is
bounded from above on $\Omega$ and strictly plurisubharmonic near $z$. Each such $\varphi$ is a natural candidate for a weight
function to be used in the construction of non-trivial holomorphic functions in $L^{2}(\Omega)$ by way of H\"ormander's method for
solving the $\dbar$-equation. In fact, we obtain the following result. 
\begin{theorem}\label{T:suffinfinite}
  Let $\Omega \subset \mathbb{C}^n$ be a pseudoconvex domain such that $\Omega\setminus\mathfrak{c}'
  (\Omega)\neq\varnothing$. Then $\dim A^{2}(\Omega)=\infty$.
\end{theorem}

\medskip

The assumption of the above theorem may be verified on a large class of domains.
As a means to verify this condition for a given domain $\Omega$, we introduce the notion of a (local) peak point for the family 
$\mathcal{PSH}^0(\Omega)$ of continuous plurisubharmonic functions on $\Omega$.
A point $p \in b\Omega$ is said to be a (local) peak point for $\mathcal{PSH}^0(\Omega)$ if there exists a continuous plurisubharmonic function $\varphi \colon U \to \mathbb{R}$ on a one-sided open neighbourhood $U \subset \Omega$ of $p$ such that $\varphi^\ast(p) = 0$ and $\varphi^\ast|_{\bar{U} \setminus \{p\}} < 0$, where $\varphi^\ast(z) \coloneqq \limsup_{z' \to z} \varphi(z')$. 

\begin{theorem}\label{T:suffcore}
Let $\Omega \subset \mathbb{C}^n$ be a domain. Then $\Omega \setminus \mathfrak{c}'(\Omega) 
\neq \varnothing$ holds true in the following cases:
 \begin{itemize}
   \item[(i)] There exists a point $p \in b\Omega$ which is a (local) peak point for $\mathcal{PSH}^0(\Omega)$.
   \item[(ii)] There exists a point $p \in b\Omega$ 
   near which $b\Omega$ is $\mathcal{C}^{\infty}$-smooth and pseudoconvex and such
    that it is of finite type in the sense of D'Angelo. 
   \item[(iii)] There exists a point $p \in b\Omega$ which is strictly pseudoconvex.
 \end{itemize}
\end{theorem}

Observe that Theorems \ref{T:suffinfinite} and \ref{T:suffcore} immediately imply the following result, which gives a negative answer to a question raised in \cite[Question 9]{HST14}.
\begin{corollary} \label{C:fatoubieberbach}
A Fatou--Bieberbach domain cannot have a strictly pseudoconvex neighbourhood.
\end{corollary}

A domain which does not satisfy the conditions of Theorem \ref{T:suffcore} is given by $\Omega = D\times\mathbb{C}$ for any domain $D \subset \mathbb{C}$. Indeed, it follows immediately from Liouville's theorem that $\Omega = \mathfrak{c}'(\Omega)$. Note that for every $h\in A^{2}(\Omega)$ and almost every ${z} \in D$ it holds that $h(z,\cdot)\in A^{2}(\mathbb{C}) = \{0\}$. Continuity then yields $h\equiv 0$ and hence $A^{2}(\Omega)=0$. It turns out that this example is, in a sense, typical, namely, the following statement is true.

\begin{theorem} \label{T:finitedim}
  Let $\Omega \subset \mathbb{C}^n$ be a pseudoconvex domain. If $\dim A^2(\Omega) < \infty$, 
  then $\bar{\Omega}^c$ is 1-pseudoconvex. In particular, if $n=2$, then $\bar{\Omega}^c$ is pseudoconvex.
\end{theorem}

Conditions (i)--(iii) in Theorem \ref{T:suffcore} are not necessary. Indeed, in Section \ref{S:example} we will give examples of pseudoconvex domains $\Omega \subset \mathbb{C}^2$ with smooth Levi-flat boundary such that $\Omega\setminus\mathfrak{c}'(\Omega)\neq\varnothing$ and $\dim A^2(\Omega) =\infty$.

\section{Proofs}
\subsection*{Proof of Theorem \ref{T:suffinfinite}}
The following variant of H\"ormander's Theorem 2.2.1' in \cite{Hormander65} will be used in the proof of Theorem \ref{T:suffinfinite}.

\begin{theorem}\label{T:Hormander}
Let $\Omega \subset \mathbb{C}^n$ be a pseudoconvex domain and let $\Phi \colon \Omega \to [-\infty, \infty)$ be plurisubharmonic. Assume that
\begin{enumerate}
 \item[(i)] $U \subset \Omega$ is open such that $\Phi - c\left|\,\cdot\,\right|^2$ is plurisubharmonic on $U$ for some constant $c>0$, and
 \item[(ii)] $v \in L^2_{(0,1)}(\Omega, \Phi)$ is a $\mathcal{C}^\infty$-smooth form such that $\bar{\partial} v = 0$ and $\supp v \subset U$.
\end{enumerate}
Then there exists a $\mathcal{C}^\infty$-smooth function $u \colon \Omega \to \mathbb{R}$ such that $\bar{\partial} u = v$ and
 \[ \int_\Omega \lvert u \rvert^2 e^{-\Phi} dV \le \frac{1}{c} \int_\Omega  \lvert v \rvert^2 e^{-\Phi} dV.  \]
\end{theorem}
Here $L^2_{(0,1)}(\Omega, \Phi)$ denotes the space of $(0,1)$-forms on $\Omega$ whose coefficients have finite $L^2$-norm with respect to the weighted measure $e^{-\Phi}dV$.
\medskip

Some remarks are in order to explain the discrepancies of Theorem \ref{T:Hormander} to the cited Theorem 2.2.1' in \cite{Hormander65}.

(a) In Theorem 2.2.1' of \cite{Hormander65} strict plurisubharmonicity of the weight function on all of $\Omega$ is assumed. The proof under the above weaker hypothesis follows from a slight variation of H\"ormander's proof. That is, to show (1.1.7) in \cite{Hormander65} for $g$ with $\supp g \subset U$, first write  $\chi_U$ for the characteristic function of $U$, so that
\begin{align}\label{E:Hormanderref}  
  |(g,f)_{2}|\leq\frac{1}{\sqrt{c}}\left\|g\right\|_{2}\cdot \sqrt{c}\left\|\chi_U\cdot f\right\|_{2}
\end{align}  
follows. Then use Theorem 2.1.4 in \cite{Hormander65}  in conjunction with the plurisubharmonicity conditions of the weight function  to estimate the last term on the right hand side of \eqref{E:Hormanderref}. This yields (1.1.7) (with $Af=\chi_{U}\cdot f/\sqrt{c}$) in \cite{Hormander65}. Now proceed with H\"ormander's proof as in \cite[Theorem 2.2.1']{Hormander65}.

(b) Theorem 2.2.1' yields a solution $u$ in $L^{2}(\Omega,\Phi)$. Ellipticity of $\dbar$ on functions means that $u$ is smooth since the data $v$ in the above Theorem \ref{T:Hormander} is smooth.

\medskip

With Theorem \ref{T:Hormander} in hand, non-trivial holomorphic functions in $L^{2}$ may be constructed under the hypothesis of Theorem \ref{T:suffinfinite} as follows. 

\begin{lemma} \label{L:linearlyindependent}
  Let $\Omega \subset \mathbb{C}^n$ be a pseudoconvex domain. Let $U \subset \Omega$ be open and assume that 
  there exists a plurisubharmonic function $\varphi \colon \Omega \to [-\infty, 0)$  such that $\varphi$ is strictly 
  plurisubharmonic on $U$ and $\nu(\varphi,z) = 0$ for every $z \in U$. Then for every finite sequence
  of pairwise distinct points $z_1, z_2, \ldots, z_N \in U$  
  there exists a function $h \in A^2(\Omega)$ such that $h(z_j) = 0$ for every $j\in\{1, 2, \ldots, N-1\}$ and $h(z_N) = 1$.

\end{lemma}
\begin{proof}[Proof of Lemma \ref{L:linearlyindependent}]
  Choose $\varepsilon \in (0,1)$ so small that $|z_j-z_k| > 2\varepsilon$ for $j \neq k$ and 
  such that $\mathbb{B}(z_j,\varepsilon) 
  \subset U$ for every $j\in\{1, \dots, N\}$. Moreover, let $\chi \colon \mathbb{C}^n \to [0,1]$ 
  be a smooth function such that $\chi (z)=1$ 
  for $|z|< \frac{\varepsilon}{4}$ and $\chi (z)=0$ for $|z|> \frac{3\varepsilon}{4}$. 
  
  \medskip
  
  We define the $\dbar$ data 
  $$ v \coloneqq \dbar \chi (z-z_N). $$
  Further, we introduce a weight function
  $$\Phi(z) \coloneqq K\varphi(z)+\sum_{j=1}^{N}\Bigl(2n \chi (z-z_j) \log |z-z_j| \Bigr),$$
  where the constant $K>0$ is chosen so large that $\Phi(z)$ is plurisubharmonic on $\Omega$ and
  $\Phi(z) - \lvert z \rvert^2$ is plurisubharmonic on $\mathbb{B}(z_N, \varepsilon)$. 
  Observe that $e^{-\Phi}$ is locally integrable near the support of $v$. Indeed,
  $$\nu(\varphi,z):=\liminf_{w\to z}\frac{\varphi(w)}{\log|w-z|}=0$$
  holds for $z\in U$ by assumption. Hence
  $\nu(K\varphi,\cdot) = K\nu(\varphi,\cdot) \equiv 0$ on $U$. It then follows 
  that $e^{-K\varphi} \in L^1_{\mathrm{loc}}(U)$ for every $K>0$, see \cite[Proposition 7.1]{Skoda72}.
  Thus $v \in L^2_{(0,1)}(\Omega, \Phi)$ and, by Theorem \ref{T:Hormander}, 
  it follows that there exists a solution 
  $u \in \mathcal{C}^{\infty}(\Omega)$ to $\dbar u = v$ such that 
  \begin{equation*} \label{equ:L2estimate} 
  \int_{\Omega} |u|^2e^{-\Phi} dV \leq \int_{\Omega} |v|^2e^{-\Phi} dV < \infty. 
  \end{equation*}
  Since  $e^{-\Phi}$ is not locally integrable at any of the points $z_1, z_2, \ldots, z_N$, 
  we conclude that $u(z_j)=0$ for every $j\in\{1, \dots, N\}$.
  Further, observe that $\Phi<0$, hence $u \in L^2(\Omega)$. 
  It follows that $h \coloneqq \chi(z-z_N) - u$ 
  is a function as desired.
\end{proof}

The proof of Theorem \ref{T:suffinfinite} now follows easily.

\begin{proof}[Proof of Theorem  \ref{T:suffinfinite}]
  Assume that $\Omega\setminus\mathfrak{c}'(\Omega)\neq\varnothing$. 
  Since $\mathfrak{c}'(\Omega)$ is 
  relatively closed in $\Omega$, there
  exists a non-empty open set $U \subset \Omega \setminus \mathfrak{c}'(\Omega)$. 
  Fix an arbitrary number $N \in \mathbb{N}$ and let 
  $z_1, z_2, \ldots, z_N \in U$ be pairwise distinct. By Lemma \ref{L:linearlyindependent}, 
  it follows that there exists 
  functions $h_1, h_2, \ldots, h_N \in A^2(\Omega)$ such that $h_k(z_j) = \delta_{jk}$. In 
  particular, the family 
  $\{h_j\}_{j=1}^N \subset A^2(\Omega)$ is linearly independent. 
  Since $N \in \mathbb{N}$ 
  was arbitrary, the proof is complete.  
\end{proof}

\begin{remark}\label{R:Remark}
 (a) Recall that the Bergman kernel function, $K_{\Omega}$, associated to $\Omega$, evaluated at $z\in\Omega$ is the maximum value of $|f(z)|^{2}$ for $f\in A^{2}(\Omega)$ with $\int_\Omega \lvert f \rvert^2 dV = 1$. Hence Lemma \ref{L:linearlyindependent} implies that $K_{\Omega}(z)>0$ whenever $z\in\Omega\setminus\mathfrak{c}'(\Omega)$. In \cite{Herbort83} 
 the domain $D=\{z\in\mathbb{C}^{3}: \text{Re}(z_{1})+|z_{2}|^{6}+|z_{2}|^{2}|z_{3}|^{2}<0\}$ was considered.  It was shown that $K_{D}(z)=0$ whenever $z_{2}=0$. Thus, $\mathfrak{c}'(D)\neq\varnothing$. It follows from Lemma \ref{L:finitetype} below that $D\setminus\mathfrak{c}'(D)\neq\varnothing$, so that $A^{2}(D)$ is infinite dimensional.
 
 (b) By using the same technique as in the proof of Lemma \ref{L:linearlyindependent}, it is easy to also construct, 
 for every $p \in U$ and every $j = 1, \ldots, n$, a function $f_{j} \in A^2(\Omega)$ such that $f_{j}(p) = 0$ and 
 $(\partial f_{j} / \partial z_k)(p) = \delta_{jk}$ (for fixed $j$, replace $v$ by $\bar{\partial}[\chi(z-p)(z_j-p_j)]$ and in 
 the definition of $\Phi$ change the factor $2n$ into $2n+1$).
 Thus it follows from a result due to Kobayashi \cite{Kobayashi59} that every pseudoconvex domain $\Omega \subset
 \mathbb{C}^n$ which satisfies the condition $\mathfrak{c}'(\Omega) = \varnothing$ possesses a Bergman metric.

 (c) Let $\varphi \colon \Omega \to [-\infty,0)$ be plurisubharmonic and assume that $\varphi$ is strictly plurisubharmonic
 on some open set $U \subset \Omega$. If $\varphi$ is locally bounded on $U$, then $\Omega \setminus \mathfrak{c}'(\Omega) 
 \neq \varnothing$. Indeed, it follows immediately that $\psi \coloneqq
 -\log(1-\varphi)$ is a bounded from above function in $\mathcal{PSH}'(\Omega)$ such that $\psi$ is strictly plurisubharmonic
 on $U$.
 
 (d) The advantage of considering $\mathfrak{c}'(\Omega)$ for a given domain $\Omega$   instead of the smooth core, $\mathfrak{c}(\Omega)$, associated to the family  of $\mathcal{C}^{\infty}$-smooth, plurisubharmonic functions  may now be observed. To that end, consider  the domain 
 $$\Omega=\left\{(z,w)\in\mathbb{C}^{2}:\log|z|+(|z|^{2}
 +|w|^{2})<C\right\}$$
 for some constant  $C$.
 Set $\varphi(z,w)=\log|z|+\frac{1}{2}(|z|^{2}+|w|^{2})-C$ and notice that $\varphi$ is plurisubharmonic and bounded from above by $-\frac{1}{2}(|z|^{2}+|w|^{2})$ on $\Omega$. By (c), it follows that 
 $\psi:=-\log(1-\varphi)\in\mathcal{PSH}'(\Omega)$. Moreover, 
 $\psi+\log(1+\frac{1}{2}(|z|^{2}+|w|^{2}))$ belongs to $\mathcal{PSH}'(\Omega)$, is negative and strictly plurisubharmonic on $\Omega$ (for the last assertion, observe that the Levi form of the second term is positive definit on $\mathbb{C}^2$). That is, $\mathfrak{c}'(\Omega)=\varnothing$. Thus, by (b), $\Omega$ possesses a Bergman metric. On the other hand, it follows from Liouville's Theorem that  the smooth core, $\mathfrak{c}(\Omega)$, must contain the complex line $\{0\}\times\mathbb{C}$. In fact, it may be shown that $\mathfrak{c}(\Omega)=\{0\}\times\mathbb{C}$, see \cite[Example 5 in Section 3]{HST14}. Hence the analogous version of 
Theorem~\ref{T:suffinfinite} for $\mathfrak{c}(\Omega)$ could not yield that $\Omega$ possesses a Bergman metric.
 \end{remark}

\medskip

\subsection*{Proof of Theorem \ref{T:suffcore}}

\begin{lemma} \label{L:somewherepeak}
Let $\Omega \subset \mathbb{C}^n$ be a domain such that at least one boundary point is a local peak point for $\mathcal{PSH}^0(\Omega)$. Then $\Omega \setminus \mathfrak{c}'(\Omega) \neq \varnothing$.
\end{lemma}
\begin{proof}
Let $\varphi \colon U \to \mathbb{R}$ be continuous and plurisubharmonic such that $\varphi^\ast(p) = 0$ and $\varphi^\ast|_{\bar{U} \setminus \{p\}} < 0$, where $U \subset \Omega$ is a one-sided open neighbourhood of a suitable point $p \in b\Omega$. Fix a constant $r>0$ such that $\Omega \cap \mathbb{B}(p,r) \subset U$, let $m \coloneqq \sup \{\varphi^\ast(z) : z \in \bar{\Omega} \cap b\mathbb{B}(p,r)\} < 0$ and choose $\varepsilon > 0$ so small that $\varepsilon r^2 < -m/2$. Then the trivial extension of $\max(\varphi|_{\Omega \cap \mathbb{B}(p,r)} + \varepsilon|z-p|^2, m/2)$ to $\Omega$ by $m/2$ defines a function $\psi \in \mathcal{PSH}'(\Omega)$ that is bounded from above on $\Omega$ and strictly plurisubharmonic on some open subset of $\Omega \cap \mathbb{B}(p,r)$. Thus $\Omega \setminus \mathfrak{c}'(\Omega) \neq \varnothing$.
\end{proof}

\medskip

It follows from work of Cho in \cite{Cho92} that each smooth boundary point of finite type in the sense of D'Angelo \cite{DAn82} is a peak point for $\mathcal{PSH}^0(\Omega)$. In particular, we have the following result.
  
\begin{lemma}\label{L:finitetype}
  Let $\Omega \subset \mathbb{C}^{n}$ be a domain
  which is $\mathcal{C}^\infty$-smooth and pseudoconvex near some boundary point $p$. If $b\Omega$ is of finite type in the sense of D'Angelo at $p$, then  $\Omega\setminus\mathfrak{c}'(\Omega)\neq\varnothing$.
\end{lemma}
 
\begin{proof} 
  Let $p\in b\Omega$ such that $b\Omega$ is pseudoconvex and of finite type at $p$, $r$ a smooth defining function for
  $\Omega$ near $p$. 
    It follows from Theorem 3 and its proof in \cite{Cho92}, that
  $\Omega$ admits a pseudoconvex support surface near $p$. That is, 
  there exists a $\mathcal{C}^{2}$-function $\rho: U \to \mathbb{R}$ on an open neighbourhood $U$ of $p$ such that
  \begin{itemize}
    \item[(a)] the gradient of $\rho$ does not vanish on $\{z\in U:\rho(z)=0\}$,
     \vspace{0.1cm}
     \item[(b)] $\rho(z) \leq r(z)$ for $z\in U$,
      \vspace{0.1cm}
    \item[(c)] $\rho(p) = 0$ and $\rho(z) \leq -c|z-p|^{K}$ for $z\in\bar{\Omega}\cap U$ 
    and some constants $c,K>0$,
      \vspace{0.1cm}
    \item[(d)] $D:=\{z\in U:\rho(z)<0\}$ is pseudoconvex near $p$.
  \end{itemize}
  \vspace{0.1cm}
  Note that (b) and (c) imply that $\Omega\cap U$ is contained in $D$ and 
  $(b\Omega \cap bD) \cap U=\{p\}$. Furthermore, it follows from (a)
  and (d) that the domain $D$ is $\mathcal{C}^{2}$-bounded 
  and pseudoconvex near $p$.
  By a result of Diederich--Forn\ae ss \cite{DieFor77-2}, there exist constants
  $L,\eta>0$ such that the function $\psi := - (e^{-L|z-p|^2} \dist(z,bD) )^\eta$ on $D$
  is strictly plurisubharmonic near $p$. Since $\psi^\ast(p)=0$ and $\psi(z)<0$ for all 
  $z\in\bar{\Omega}\cap D$, it follows that $p$ is a local peak point for $\mathcal{PSH}^0(\Omega)$. 
  The claim thus follows from Lemma \ref{L:somewherepeak}. 
\end{proof}
 
\medskip 

\begin{lemma}\label{C:somewherestrict}
  Let $\Omega\subset\mathbb{C}^{n}$ be a domain which is strictly pseudoconvex near 
  at least one
  boundary point. Then $\Omega\setminus\mathfrak{c}'(\Omega)\neq\varnothing$.
\end{lemma}

\begin{proof}
  Let $p \in b\Omega$ be such that $\Omega$ is strictly pseudoconvex at $p$. 
  Then there exist a neighbourhood $U' \subset \mathbb{C}^n$ of $p$ and a 
  continuous strictly plurisubharmonic function $\varphi' \colon U' \to \mathbb{R}$ 
  such that $\Omega \cap U' = \{\varphi' < 0\}$. 
  Let $U \Subset U'$ be another neighbourhood of $p$ and set $\varphi(z) \coloneqq \varphi'(z) -\varepsilon \lvert z-p \rvert^2$.
  Then, for $\varepsilon > 0$ small enough, the function $\varphi$ is plurisubharmonic on $U$ such that $\varphi(p) = 0$
  and $\varphi(z) < 0$ for every $z \in (\bar{\Omega} \cap U) \setminus \{p\}$, i.e., $p$ is a local peak point for 
  $\mathcal{PSH}^0(\Omega)$. As before, the claim thus follows from Lemma \ref{L:somewherepeak}.
\end{proof}
 
\medskip

\subsection*{Proof of Theorem \ref{T:finitedim}}
 For each $k \in \mathbb{N}$, let $\Delta^k \coloneqq \{z \in \mathbb{C}^k : \lvert z \rvert_\infty < 1\}$, 
 where $\lvert z \rvert_\infty = \max_{1 \le j \le k} \lvert z_j \rvert$. If $\bar{\Omega}^c \subset \mathbb{C}^n$ is not 
 1-pseudoconvex, then, by definition, there exists a Hartogs figure 
 $$ H = \big\{(z',z_n) \in \Delta^{n-1} \times \Delta :  \lvert z' \rvert_\infty > r_1 \text{ or } \lvert z_n \rvert < r_2\big\}, $$
 where $r_1,r_2 \in (0,1)$, and an injective holomorphic mapping $F \colon \Delta^n \to \mathbb{C}^n$ such that
 $F(H) \subset \bar{\Omega}^c$ but $F(\Delta^n) \cap \bar{\Omega} \neq \varnothing$ (see, for example, \cite{Riemenschneider67}). 
 Without loss of generality we can assume that 
 $F$ is defined in a neighbourhood of $\bar{\Delta}^n$. Set $U \coloneqq \Omega \cap F(\Delta^n)$, and, for fixed $\varepsilon > 0$, 
 let $\varphi_1 \colon \mathbb{C}^{n-1} \times \mathbb{C}^\ast \to \mathbb{R}$ be the strictly
 plurisubharmonic function given by $\varphi_1(z) \coloneqq -\log\lvert z_n \rvert + \varepsilon\lvert z \rvert^2$. If $\varepsilon$
 is chosen small enough, then $\varphi_2 \coloneqq \varphi_1 \circ F^{-1}$ attains a local maximum along $\bar{U}$ in some point $p \in
 b\Omega \cap F(\Delta^n)$ (see the proof of Theorem 3.2 in \cite{HST14} for details of this argument). Then $\varphi(z) \coloneqq \varphi_2(z) - (\varepsilon/2)\lvert z-p \rvert^2$ is plurisubharmonic 
 in a neighbourhood of $\bar{U}$ such that $\varphi(p) > \varphi(z)$ for every $z \in \bar{U} \setminus \{p\}$. In particular, $p$ is a 
 local peak point for $\mathcal{PSH}^0(\Omega)$. Hence it follows from Lemma \ref{L:somewherepeak} that $\dim A^2(\Omega) = \infty$.
 This completes the proof of Theorem \ref{T:finitedim}. \hfill $\Box$

\medskip

\section{Examples}\label{S:example}
In this section, two examples of domains in $\mathbb{C}^{2}$ with Levi-flat boundary are given for which the dimension of the Bergman space is infinite. 

\begin{example}
 The following construction mimics Example 7 of \cite{HST14}, where an unbounded strictly pseudoconvex domain with empty core was constructed.
 It follows from work of Globevnik, see Corollary 1.1 in \cite{Gl98}, that there exists a Fatou--Bieberbach 
 domain $\Omega' \subset \mathbb{C}^2_{z,w}$ such that
 $$(0,0) \in \Omega' \cap \{w=0\} \subset 
 \Delta(0,1) \times \{0\}\;\;\text{and}\;\;\overline{\Omega' \cap \{w=0\}} = \bar{\Omega}' \cap \{w=0\}.$$
 Fix $\varepsilon > 0$ and let $h \colon \bar{\Delta}(0,1+\varepsilon) \to \mathbb{R}$ 
 be a harmonic function such that 
 $$\big\{(z,w) \in \mathbb{C}^2 : |z| = 1+\varepsilon, |w| \le e^{h(z)}\big\} \subset \mathbb{C}^2 
 \setminus \bar{\Omega}';$$
 for example, we may choose $h(z) \equiv -C$ for some large enough constant $C>0$.
 Moreover, fix a biholomorphism $\Phi \colon \mathbb{C}^2 \xrightarrow{\sim} \Omega'$. Now define an open set
 \begin{align*}
    \Omega \coloneqq \Phi^{-1}\big(\Omega' \cap \big\{(z,w) \in \mathbb{C}^2 : |z| < 1+
    \varepsilon, |w| < e^{h(z)}\big\}\big),
  \end{align*}
 and, if necessary, replace $\Omega$ by its connected component containing the point $\Phi^{-1}((0,0))$. Then $\Omega$ is an 
 unbounded domain  with smooth and Levi-flat boundary. In fact, $b\Phi(\Omega) \cap \Omega'$ is an 
 open subset of the Levi-flat hypersurface $$\big\{(z,w) \in \mathbb{C}^2 : 
 |z| < 1+\varepsilon, |w| = e^{h(z)}\big\},$$ and $\lim_{z \to z_0} \Phi^{-1}(z) = \infty$ for every 
 $z_0 \in b\Phi(\Omega) \cap b\Omega'$.  However, observe that $\Phi(\Omega) \subset \mathbb{C}^2$ is bounded, 
 and thus $\mathfrak{c}'(\Omega) = \varnothing$. In particular, $\dim A^2(\Omega) = \infty$.
\end{example}    

\begin{example}
For every number $a \in \mathbb{R}$, let $[a] \in \mathbb{Z}$ denote the integral part of $a$, and let $\{a\} \coloneqq a - [a] \in [0,1)$ be the fractional part of $a$. Fix a sequence $a_1, a_2, a_3, \ldots$ of constants $a_j \ge 0$ and assume that there exists a finite subset $J \subset \mathbb{N}$ such that 
\begin{enumerate}
  \item[(a)] $\#J \ge 2$, 
    \vspace{0.1cm}
  \item[(b)] $\{a_j\} > 0$ for every $j \in J$,
    \vspace{0.1cm}
  \item[(c)] $\{\sum_{j \in J} a_j\} > 0$.
\end{enumerate} 
At the same time, fix a sequence $z_1, z_2, z_3 \ldots$ of points $z_j \in \mathbb{C}$ such that
\begin{enumerate}
  \item[(d)] $z_j \neq z_k$ for $j \neq k$, 
    \vspace{0.1cm}
  \item[(e)] $\{z_j\}_{j=1}^\infty \subset \mathbb{C}$ is discrete, and
    \vspace{0.1cm}
  \item[(f)] $\sum_{j=1}^\infty a_j\log\lvert z-z_j \rvert$ converges and is smooth on $\mathbb{C} \setminus \{z_j\}_{j=1}^\infty$.
\end{enumerate} 
For generic $C \in \mathbb{R}$, consider the function $\varphi(z) \coloneqq C + \sum_{j=1}^\infty a_j\log\lvert z-z_j \rvert$. Then
\[ \Omega \coloneqq \big\{(z,w) \in \mathbb{C}^2 : \lvert w \rvert < e^{-\varphi(z)} \big\} \]
is an unbounded domain with smooth and Levi-flat boundary. Moreover, by Theorem 4.1 of \cite{Juc12}, it follows that $\dim A^2(\Omega) = \infty$.
\end{example}

\bibliographystyle{acm}
\bibliography{References}

\end{document}